\documentclass[11pt]{amsart}
\usepackage[margin=1.0in]{geometry}

\usepackage{amsmath,amssymb,amsxtra,amsthm,hyperref}
\usepackage{tikz}
\usepackage{multicol}



\newcommand{\lspan} {\operatorname{span}}

\newcommand{\condref}[1] {\hyperref[cond.#1]{(#1)}}

\newcommand{\Aut}{\operatorname{Aut}}

\newcommand{\spn}{\operatorname{span}}

\newcommand{\ca}{\mathrm{C}^*}

\newcommand{\A}{{\mathcal{A}}}
  \newcommand{\B}{{\mathcal{B}}}

  \newcommand{\E}{{\mathcal{E}}}

  \newcommand{\N}{{\mathcal{N}}}
\renewcommand{\O}{{\mathcal{O}}}

  \newcommand{\T}{{\mathcal{T}}}

\newcommand{\cA}{\mathcal{A}}
\newcommand{\cB}{\mathcal{B}}

\newcommand{\cE}{\mathcal{E}}

\newcommand{\cK}{\mathcal{K}}
\newcommand{\cL}{\mathcal{L}}

\newcommand{\cN}{\mathcal{N}}
\newcommand{\cO}{\mathcal{O}}

\newcommand{\cT}{\mathcal{T}}

\newcommand{\fA}{\mathfrak{A}}

\newcommand{\bC}{{\mathbb{C}}}

\newcommand{\bN}{{\mathbb{N}}}

\newcommand{\foral}{\text{ for all }}

\newcommand{\qforal}{\quad\text{for all}\quad}

\newcommand{\ZS}{Zappa-Sz\'{e}p }

\newtheorem{theorem}{Theorem}[section]

\newtheorem{corollary}[theorem]{Corollary}
\newtheorem{proposition}[theorem]{Proposition}    %%%%%%%%%%%%%%%%%%%%%%%%%5

\theoremstyle{definition}
\newtheorem{rem}[theorem]{Remark}

\newtheorem{defn}[theorem]{Definition}

\newtheorem{eg}[theorem]{Example}

\theoremstyle{definition}
\newtheorem{definition}[theorem]{Definition}

\newtheorem{remark}[theorem]{Remark}

\numberwithin{equation}{section}

\begin{document}

\title{Zappa-Sz\'{e}p actions of groups on product systems}
%{Zappa-Sz\'{e}p products of product systems by groups}

\author{Boyu Li}
\address{Department of Mathematics and Statistics, University of Victoria, Victoria, B.C. V8W 3R4}
\email{boyuli@uvic.ca}

\author{Dilian Yang}
\address{Department of Mathematics and Statistics, University of Windsor, Windsor, ON.  N9B 3P4}
\email{dyang@uwindsor.ca}
\date{\today}

\thanks{The first author was supported by a fellowship of Pacific Institute for the Mathematical Sciences. The second author was partially supported by an NSERC Discovery Grant 808235.}

\subjclass[2010]{46L55, 46L05, 20M99.}
\keywords{Zappa-Sz\'{e}p action, product system, right LCM semigroup, Nica-Toeplitz representation}

\begin{abstract}  
%%%%%
Let $G$ be a group and $X$ be a product system over a semigroup $P$. Suppose $G$ has a left action on $P$ and $P$ has a right action on $G$, so 
that one can form a Zappa-Sz\'ep product $P\bowtie G$. We define a Zappa-Sz\'ep action of $G$ on $X$ to be a collection of functions 
on $X$ that are compatible with both actions from $P\bowtie G$ in a certain sense. Given a Zappa-Sz\'ep action of $G$ on $X$, we construct 
a new product system $X\bowtie G$ over $P\bowtie G$, called the Zappa-Sz\'ep product of $X$ by $G$. We then associate to $X\bowtie G$ 
several universal C*-algebras and prove their respective Hao-Ng type isomorphisms. 
A special case of interest is when a Zappa-Sz\'{e}p action is homogeneous. This case naturally 
generalizes group actions on product systems in the literature. For this case, besides the Zappa-Sz\'ep product system $X\bowtie G$,
one can also construct a new type of Zappa-Sz\'{e}p product $X \widetilde\bowtie G$ over $P$. Some essential differences arise between these two 
types of Zappa-Sz\'ep product systems and their associated C*-algebras. 
\end{abstract}

\maketitle

\section{Introduction}

In group theory, a \ZS product of two groups $G$ and $H$ generalizes a semi-direct product by encoding a two-way action between $G$ and $H$. In addition to a left action of $G$ on $H$, the \ZS product encodes an 
additional right action of $H$ on $G$. An analogue of semi-direct products in operator algebra is the crossed product construction 
arising from various dynamical systems. In its simplest form, a group $G$ act on a C*-algebra $\cA$ by $*$-automorphisms, in a similar fashion as in a semi-direct product. We seek to extend the \ZS type construction into the field of operator algebras, which would naturally generalize the crossed product construction. 

To construct a \ZS type structure in operator algebras, there are two key ingredients. First, a \ZS product of an operator algebra $\cA$ and a group $G$ requires a left action of $G$ on $\cA$ and a right action of $\cA$ on $G$. 
The right $\cA$ action on $G$ requires a grading on $\cA$. In the operator algebra literature, there are two natural ways of putting a grading: either via Fell bundles graded by groupoids, 
or via product systems graded by semigroups. The \ZS product of a Fell bundle by a groupoid is recently studied in \cite{DL2020}. This paper aims to study \ZS products of product systems by groups. The second key ingredient is an appropriate replacement of the group action in a dynamical system. A \ZS product of a product system by a group $G$ needs to encode a two-way action between the product system and the group in the scenario. This requires the group action on the product system to be compatible with the product system action on the group. Finding a right notion of a \ZS product with compatible actions in the context of product systems is a key in our construction. 

Let $P$ be a semigroup and $G$ a group. Suppose $G$ has a left action on $P$ and $P$ has a right action on $G$, so that one can form a \ZS product $P\bowtie G$. Let $X$ be a product system over $P$. 
We introduce a notion of the \ZS action of $G$ on $X$ in Definition \ref{D:beta}. Given a \ZS action of $G$ on $X$, we define a \ZS product of $X$ by $G$. 
We show in Theorem \ref{T:ZSP} that it is a product system over the \ZS product semigroup $P\bowtie G$. This product system is denoted by $X\bowtie G$. We then consider the scenario where the $G$-action on $P$ is homogeneous. In such a case, it turns out that we can naturally define another \ZS type product, denoted by $X\widetilde\bowtie G$, which is a product system over the same semigroup $P$ (Theorem \ref{T:prop.Z}). 

In Section \ref{S:main}, we study covariant representations of \ZS actions of groups on product systems and their associated C*-algebras. A covariant representation of a \ZS action $(X,G,\beta)$ 
is a pair $(\psi, U)$ consisting of a Toeplitz representation of $X$ and a unitary representation $U$ of $G$ that satisfy a covariant relation.
First, we exhibit a one-to-one correspondence between the set of all covariant representations $(\psi, U)$ of $(X, G,\beta)$ and the set of all Toeplitz representations $\Psi$ of $X\bowtie G$
 (Theorem \ref{T:Upsi}). 
As an immediate consequence, we obtain a Hao-Ng isomorphism theorem: $\cT_{X\bowtie G}\cong \cT_X\bowtie G$ (Corollary \ref{C:HaoNgT}). 
Furthermore, we show that $\psi$ is Cuntz-Pimsner covariant if 
and only if so is $\Psi$, and so the Hao-Ng isomorphism also holds true for the associated Cuntz-Pimsner C*-algebras: $\cO_{X\bowtie G}\cong \cO_X\bowtie G$ (Theorem \ref{thm.cp} and Corollary \ref{C:HaoNgC}). 
Moreover, if $(X,G,\beta)$ is homogeneous, then 
$\cT_{X\bowtie G}\cong \cT_X\bowtie G\cong \cT_{X\widetilde\bowtie G}$ (Theorem \ref{T:Upsi.homo} and Corollary \ref{C:HaoNgT}). 
However, we do not know whether one has $\cO_{X\widetilde\bowtie G}\cong \cO_X\bowtie G$. Indeed, this is still unknown 
even in the special case of semi-direct products. This is related to an open problem of Hao-Ng in the literature. Lastly, when the semigroup $P$ is right LCM and $X$ is compactly aligned, we show that $\psi$ is 
Nica covariant if and only if so is $\Psi$ (Theorem \ref{T:Ncov}). As a result, we have the Hao-Ng isomorphism theorem for Nica-Toeplitz algebras as well: $\cN\cT_{X\bowtie G}\cong \cN\cT_X\bowtie G$ (Corollary \ref{C:HaoNgNT}).

Finally, in Section \ref{S:EX}, we present some examples of \ZS actions of groups on product systems and their C*-algebras.

\section{Preliminaries}

In this section, we provide some necessary background for later use. 

\subsection{\ZS products}  

Let $P$ be a (discrete) semigroup and $G$ be a (discrete) group. 
By convention, in this paper, we always assume that \textsf{a semigroup $P$ has an identity}, written as $e$, unless otherwise specified. 
To define a \ZS product semigroup of $P$ and $G$, we first need two actions between $P$ and $G$, given by
\begin{enumerate}
\item a left $G$-action on $P$: $G\times P \to P$, denoted by $(g,p)\mapsto g\cdot p$, and 
\item a right $P$-action on $G$ (also called a restriction map): $P\times G\to G$, denoted by $(p,g)\mapsto g|_p$. 
\end{enumerate}
%Let $e$ and $e$ be the identities in $P$ and $G$ respectively. 
Suppose that the two actions satisfy the following compatibility relations:
\begin{multicols}{2}
\begin{enumerate}
\item[(ZS1)]\label{cond.ZS1} $e\cdot p=p$;
\item[(ZS2)]\label{cond.ZS2} $(gh)\cdot p=g\cdot (h\cdot p)$;
\item[(ZS3)]\label{cond.ZS3} $g\cdot e=e$;
\item[(ZS4)]\label{cond.ZS4} $g|_{e}=g$;
\item[(ZS5)]\label{cond.ZS5} $g\cdot (pq)=(g\cdot p)(g|_p \cdot q)$;
\item[(ZS6)]\label{cond.ZS6}$g|_{pq}=(g|_p)|_q$;
\item[(ZS7)]\label{cond.ZS7} $e|_p=e$;
\item[(ZS8)]\label{cond.ZS8} $(gh)|_p=g|_{h\cdot p} h|_p$.
\end{enumerate}
\end{multicols}
\noindent
Then the \ZS product semigroup $P\bowtie G$ is defined by $P\bowtie G=\{(p,g): p\in P, g\in G\}$ with multiplication $(p,g)(q,h)=(p(g\cdot q), g|_q h)$. This semigroup has an identity $(e, e)$. 

Recall that a left cancellative semigroup $P$ is called a right LCM semigroup if for any $p,q\in P$, either $pP\cap qP=\emptyset$ or $pP\cap qP=rP$ for some $r\in P$. In the case when $P$ is a right LCM semigroup, $P\bowtie G$ is known to be a right LCM semigroup as well  \cite[Lemma 3.3]{BRRW}. 

\subsection{Product systems}

We give a brief overview of product systems. One may refer to \cite{Fowler2002} for a more detailed discussion.

\begin{definition} Let $\cA$ be a unital C*-algebra and $P$ a semigroup. A \textit{product system over $P$} with coefficient $\cA$ is defined as $X=\bigsqcup_{p\in P} X_p$ consisting of $\cA$-correspondences $X_p$ and an associative multiplication $\cdot:X_p \times X_q\to X_{pq}$ such that
\begin{enumerate}
\item $X_e=\cA$ as an $\cA$-correspondence;
\item for any $p,q\in  P$, the multiplication map on $X$ extends to a unitary $M_{p,q}: X_p\otimes X_q\to X_{pq}$;
\item the left and right module multiplications by $\cA$ on $X_p$ coincides with the multiplication maps on $X_e\times X_p\to X_p$ and $X_p\times X_e\to X_p$, respectively.
\end{enumerate}
\end{definition}

Implicit in $M_{e,q}$ being unitary, $X_p$ must be essential, that is, $\lspan\{a\cdot x: a\in X_e, x\in X_p\}$ is dense in $X_p$. This assumption is absent in Fowler's original construction as he does not require $M_{e,q}$ to be unitary. Nevertheless, when the semigroup $P$ contains non-trivial units, every $X_p$ must be essential \cite[Remark 1.3]{KL2018}. Since the semigroups $P$ and $P\bowtie G$ often contain non-trivial units, it is reasonable to 
make such an assumption. 

For a C*-correspondence $X$, we use $\cL(X)$ to denote the set of all adjointable operators on $X$. It is a C*-algebra when equipped with the operator norm. For any $x,y\in X$, define the operator $\theta_{x,y}:X\to X$ by $\theta_{x,y}(z)=x\langle y,z\rangle$. It is clear that $\theta_{x,y}\in \cL(X)$, and we use $\cK(X)$ to denote the C*-subalgebra of $\cL(X)$ generated by $\theta_{x,y}$. The set $\cK(X)$ is also known as the generalized compact operators on $X$. 

Suppose now that $P$ is a right LCM semigroup. The notion of compactly aligned product systems, first introduced by Fowler for product systems over quasi-lattice ordered semigroups \cite{Fowler2002}, has been 
recently generalized to right LCM semigroups
 in \cite{BLS2018b, KL2018}. For any $p,q\in P$, there is a $*$-homomorphism $i_p^{pq}: \cL(X_p)\to\cL(X_{pq})$ by setting for any  $x\in X_p$ and $y\in X_q$,
\[i_p^{pq}(S)(xy)=(Sx)y.\]

\begin{definition}
\label{df.cpt.align}
We say that a product system $X$ is \emph{compactly-aligned} if for any $S\in\cK(X_p)$ and $T\in\cK(X_q)$  with $pP\cap qP=rP$, we have
\[i_p^r(S) i_q^r(T) \in \cK(X_r).\]
We shall use the notion $S\vee T := i_p^r(S) i_q^r(T)$.

\end{definition}

\subsection{Representations and C*-algebras of product systems}

Product systems are one of essential tools in the study of operator algebras graded by semigroups. In its simplest form, an $\cA$-correspondence $X$ can be viewed as a product system over $\mathbb{N}$, by setting $X_0=\cA$ and $X_n=X^{\otimes n}$. There are two natural C*-algebras associated with such a product system: the Toeplitz algebra $\cT_X$ and the Cuntz-Pimsner algebra $\cO_X$ \cite{MS1998, Pimsner1997}. Fowler generalizes these to product systems over countable semigroups \cite{Fowler2002}. 

In \cite{Nica1992}, Nica introduced the semigroup C*-algebra of a quasi-lattice ordered semigroup. It is the universal C*-algebra generated by isometric semigroup representations that satisfy a covariance condition, now known as the Nica-covariance condition. This extra covariance condition soon found an analogue in the C*-algebra related to product systems. In \cite{Fowler2002, FR1998, BLS2018b}, the Nica-Toeplitz algebra $\cN\cT_X$ is defined by imposing the extra Nica-covariance condition, thereby being a quotient of the corresponding Toeplitz algebra. Here, we give a brief overview of these three C*-algebras and their representations associated with a product system. 

\begin{definition} Let $X$ be a product system over a semigroup $P$. A \textit{(Toeplitz) representation} of $X$ on a C*-algebra $\cB$ consists of a collection of linear maps $\psi=(\psi_p)_{p\in P}$, where for each $p\in P$, $\psi_p:X_p\to \cB$, such that
\begin{enumerate}
    \item $\psi_e$ is a $*$-homomorphism of the C*-algebra $X_e$;
    \item for all $p,q\in P$ and $x\in X_p, y\in X_q$, $\psi_{p}(x)\psi_q(y)=\psi_{pq}(xy)$; and 
    \item for all $p\in P$ and $x, y\in X_p$, $\psi_p(x)^* \psi_p(y)=\psi_e(\langle x,y\rangle)$. 
\end{enumerate}
\end{definition}

Given a representation $\psi:X\to \cB$, there is a homomorphism $\psi^{(p)}: \cK(X_p)\to \cB$ satisfying $\psi^{(p)}(\theta_{x,y})=\psi_p(x)\psi_p(y)^*$. The left action of $X_e$ on $X_p$ induces a $*$-homomorphism $\phi_p: X_e\to \cL(X_p)$ by $\phi_p(a)x=a\cdot x$. 

\begin{definition} A representation $\psi$ is called \textit{Cuntz-Pimsner covariant} if for all $p\in P$ and $a\in X_e$ with $\phi_p(a)\in \cK(X_p)$, $\psi^{(p)}(\phi_p(a))=\psi_e(a)$.
\end{definition}

By \cite[Propositions 2.8 and 2.9]{Fowler2002}, there is a universal C*-algebra $\cT_X$ (resp. $\cO_X$) for Toeplitz (resp. Cuntz-Pimsner covariant) representations. Specifically, there is a universal Toeplitz representation $i_X$ (resp. a universal Cuntz-Pimsner covariant representation $j_X$) such that the following properties hold:
\begin{enumerate}
    \item The C*-algebras $\cT_X$ and $\cO_X$ are generated by $i_X$ and $j_X$ respectively. That is, $\cT_X=C^*(i_X(X))$ and $\cO_X=C^*(j_X(X))$.
    \item For every Toeplitz (resp. Cuntz-Pimsner covariant) representation $\psi$, there exists a $*$-homomorphism $\psi_*$ from $\cT_X$ (resp. $\cO_X$) to $C^*(\psi(X))$ such that $\psi=\psi_*\circ i_X$ (resp. $\psi=\psi_* \circ j_X$).
\end{enumerate}

The Toeplitz C*-algebra $\cT_X$ is often quite large as a C*-algebra. For example, for the trivial product system $X$ over $\bN^2$,
its Toeplitz representation is determined by a pair of commuting isometries. The Toeplitz algebra $\cT_X$ of this product system is thus the universal C*-algebra generated by a pair of commuting isometries, which is known to be non-nuclear \cite{Murphy1996b}. This motivated  Nica to study the semigroup C*-algebras of quasi-lattice ordered semigroups \cite{Nica1992} by imposing what is now known as the Nica-covariance condition. This condition is soon generalized to representations of product systems. In \cite[Definition 5.1]{Fowler2002}, Fowler defined the notion of compactly aligned product system and extended the Nica-covariance condition to such product systems. In \cite[Definition 6.4]{BLS2018b}, the Nica-covariance condition is further generalized to compactly aligned product systems over right LCM semigroups. We now give a brief overview of the Nica-covariance condition.

Recall that if $X$ is compactly aligned, then for any $S\in\cK(X_p)$ and $T\in\cK(X_q)$  with $pP\cap qP=rP$, 
\[S\vee T:=i_p^r(S) i_q^r(T) \in \cK(X_r).\]

\begin{definition}\label{df.NC.rep} Let $X$ be a compactly aligned product system over a right LCM semigroup $P$. A representation $\psi$ is called \emph{Nica-covariant} if for all $p,q\in P$ and $S\in\cK(X_p)$ and $T\in \cK(X_q)$, we have,
\[\psi^{(p)}(S) \psi^{(q)}(T) = \begin{cases} 
\psi^{(r)}(S \vee T) & \text{ if } pP\cap qP=rP, \\
0 & \text{ if } pP\cap qP=0.
\end{cases}
\]
\end{definition}

One can verify that this definition does not depend on the choice of $r$  (see \cite{BLS2018b}). The Nica-Toeplitz C*-algebra $\cN\cT_X$ can be then defined as the universal C*-algebra 
generated by the Nica-covariant Toeplitz representations of the product system $X$.

%%%%%%%%%%%%%
\section{Zappa-Sz\'{e}p products of Product Systems by Groups}

We first introduce the notion of \ZS action of a group on a product system. This allows us to define the \ZS product of a product system by a group, which is a product system over 
the given \ZS product semigroup. In the special case when the \ZS action is homogeneous, it turns out that we can construct another \ZS type product system, 
which is a product system over the same semigroup.

In the remaining of this section, let $P\bowtie G$ be a \ZS product of a semigroup $P$ and a group $G$, $\cA$ a unital C*-algebra, and $X=\bigsqcup_{p\in P} X_p$ a product system over $P$ with coefficient $\cA$. 

\subsection{Zappa-Sz\'{e}p actions} 

We first need a key notion of the \ZS action of $G$ on $X$. The product system $X$ naturally defines a ``$G$-restriction map" on $X$ by inheriting the $G$-restriction map on $P$. However, one has to define a $G$-action map on the product system $X$ that mimics the $G$-action on the semigroup $P$. 

\begin{definition} 
\label{D:beta} 
(i) Let $P\bowtie G$ be a \ZS product of a semigroups $P$ and a group $G$, and $X$ a product system over $P$.
A  \textit{\ZS action of $G$ on $X$} is a collection of functions $\{\beta_g\}_{g\in G}$ that satisfies the following conditions: 
\begin{enumerate}
\item[(A1)]\label{cond.A1} for each $p\in P$ and $g\in G$, $\beta_g: X_p \to X_{g\cdot p}$ is a $\mathbb{C}$-linear map;
\item[(A2)]\label{cond.A2} $\beta_g\circ\beta_h=\beta_{gh}$ for all $g,h\in G$;
\item[(A3)]\label{cond.A3} the map $\beta_{e}$ is the identity map;
\item[(A4)]\label{cond.A4} for each $g\in G$, the map $\beta_g$ is a $*$-automorphism on the C*-algebra $\cA$;
\item[(A5)]\label{cond.A5} for each $g\in G$ and $p,\, q\in P$,
\[\beta_g(xy)=\beta_g(x)\beta_{g|_p}(y) \qforal x\in X_p, \, y\in X_q;\] 
\item[(A6)]\label{cond.A6} for each $g\in G$ and $p\in P$,
\[
\langle\beta_g(x), \beta_g(y)\rangle=\beta_{g|_p}(\langle x, y\rangle)\qforal x,\, y\in X_p.
\] 
\end{enumerate}
(ii) If $\beta$ is a \ZS action of $G$ on $X$, we call the triple $(X,G,\beta)$ a \textit{\ZS system}.
\end{definition}

\smallskip
Some remarks are in order. 
\begin{rem}
\label{R:beta}
In Definition \ref{D:beta}, to be more precise, the collection $\{\beta_g\}_{g\in G}$ should be written as $\{\beta_g^p\}_{g\in G,p\in P}$, so that $\beta_g^p: X_p\to X_{g\cdot p}$. 
But it is usually clear from the context to see which $X_p$ the map $\beta_g$ acts on. So, in order to simplify our notation, we just write $\beta_g$ instead of $\beta_g^p$. 
\end{rem}

\begin{rem}
In the very special case when both the $G$-action and the $G$-restriction are trivial, that is, $g\cdot p=p$ and $g|_p=g$ for all $g\in G$ and $p\in P$, Definition \ref{D:beta} corresponds to the 
sense of a group action on a product system in \cite{DOK20}. For a C*-correspondence $X$ (that is a product system over $\bN$), our definition coincides with the group action considered in \cite[Definition 2.1]{HaoNg2008}. 
\end{rem}

\begin{rem}
It follows from \hyperref[cond.A1]{(A1)}-\hyperref[cond.A3]{(A3)} that $\beta_g^p$ is a bijection between $X_p$ and $X_{g\cdot p}$. However, $\beta_g$ is not an $\cA$-linear map between $\cA$-correspondences. In fact, the condition \hyperref[cond.A5]{(A5)} forces $\beta_g(xa)=\beta_g(x)\beta_{g|_p}(a)$ and $\beta_g(ax)=\beta_g(a) \beta_g(x)$ for $a\in \cA$ and $x\in X_p$. 
\end{rem}

\subsection{The Zappa-Sz\'{e}p product system $X\bowtie G$ of $(X,G,\beta)$} 

Let $(X,G,\beta)$ be a \ZS system as defined in Definition \ref{D:beta}. 
For each $p\in P$ and $g\in G$, set 
\[
Y_{(p,g)}:=\{x\otimes g: x\in X_p\}.
\]
Define the left and right actions of $\cA$ on $Y_{(p,g)}$ by
\[
a\cdot (x\otimes g)=(a\cdot x)\otimes g,\ (x\otimes g)\cdot a = x\beta_g(a) \otimes g,
\]
respectively, and an $\cA$-valued inner product to be 
\[
\langle x\otimes g, y\otimes g\rangle = \beta_{g^{-1}}(\langle x,y\rangle).
\]

\begin{proposition} 
With the notation same as above, $Y_{(p,g)}$ is an $\cA$-correspondence.
\end{proposition}

\begin{proof} We first verify that $Y_{(p,g)}$ is a right $\cA$-module. 
The inner product is $\cA$-linear in the second component: take any $a\in \cA$, and $x,y\in X_p$,
\begin{align*}
\langle x\otimes g, (y\otimes g) a\rangle &= \langle x\otimes g, y\beta_g(a)\otimes g\rangle \\ 
&= \beta_{g^{-1}}(\langle x, y\beta_g(a)\rangle) \\
&= \beta_{g^{-1}}(\langle x, y\rangle \beta_g(a)) \\
&= \beta_{g^{-1}}(\langle x, y\rangle) a \\
&= \langle x\otimes g, y\otimes g\rangle a.
\end{align*}
Since $\beta_g$ is a $*$-automorphism on $\cA$, one has
\begin{align*}
\langle x\otimes g, y\otimes g\rangle^* &= \beta_{g^{-1}}(\langle x,y\rangle)^* \\
&= \beta_{g^{-1}}(\langle x,y\rangle^*) \\
&= \beta_{g^{-1}}(\langle y,x\rangle) \\
&= \langle y\otimes g, x\otimes g\rangle.
\end{align*}
Finally, $\beta_g$ is also a positive map on $\cA$ because it is a $*$-automorphism. Hence
\[
\langle x\otimes g, x\otimes g\rangle = \beta_{g^{-1}}(\langle x,x\rangle)^* \geq 0.\]
Moreover, since $\beta_g$ is isometric on $\cA$,
\[\|x\otimes g\|=\|\langle x\otimes g, x\otimes g\rangle\|^{1/2}=\|\beta_g^{-1}(\langle x, x\rangle)\|^{1/2}=\|\langle x,x\rangle\|^{1/2}=\|x\|.\]
Thus the norm on $Y_{(p,g)}$ is the same as that on $X_p$. Because $X_p$ is complete under its norm, so is $Y_{(p,g)}$. Therefore $Y_{(p,g)}$ is a Hilbert $\cA$-module. One can clearly see that the left action of $\cA$ induced from $X_p$ is a left action of $\cA$ on $Y_{p,g}$. Therefore, $Y_{p,g}$ is an $\cA$-correspondence. 
\end{proof}  

Let
\[
Y:=\bigsqcup_{(p,g)\in P\bowtie G} Y_{(p,g)}.
\] 
For each $p,q\in P$ and $g,h\in G$, define a multiplication map $Y_{(p,g)}\times Y_{(q,h)}\to Y_{(p(g\cdot q), g|_q h)}$ by
\begin{align}
\label{E:M}
((x\otimes g), (y\otimes h))\mapsto (x\beta_g(y))\otimes (g|_q)h.
\end{align}
This extends to a map $M_{(p,g),(q,h)}:Y_{(p,g)}\otimes Y_{(q,h)}\to Y_{(p(g\cdot q), g|_q h)}$.

\begin{theorem} 
\label{T:ZSP}
With the multiplication maps $M_{(p,g),(q,h)}$,  $Y$ is a product system over the \ZS product $P\bowtie G$. 
\end{theorem} 

\begin{proof} First of all, the identity of $P\bowtie G$ is $(e, e)$, and one can easily check that $Y_{(e, e)}\cong \cA$ by identifying $a\otimes e\in Y_{(e,e)}$ with $a\in \cA$. For any $a\in \cA$ (that is $a\otimes e\in Y_{(e,e)}$), and the left and right $\cA$-actions on $Y_{(p,g)}$ are implemented by the multiplications:
\begin{align*}
a(x\otimes g) &= ax\otimes g = (a\otimes e)(x\otimes g), \\
(x\otimes g)a &= x\beta_g(a) \otimes g = (x\otimes g)(a\otimes e).
\end{align*}

By \condref{A1}, $\beta_g$ is a $\bC$-linear isomorphism from $X_p$ to $X_{g\cdot p}$. Therefore, if $a\in X_p$ and $b\in X_q$, 
\[a\beta_g(b)\in X_p X_{g\cdot q}\in X_{p(g\cdot q)}.\]
One can easily see that $M_{(p,g),(q,h)}$ is an $\cA$-linear map. To show that it is unitary, 
for any $x_p, u_p\in X_p$ and $y_q, v_q\in X_q$,
\begin{align*}
&\langle (x_p\otimes g)\otimes (y_q\otimes h), (u_p\otimes g)\otimes (v_q\otimes h)\rangle \\
=& 
\langle (y_q\otimes h), \langle(x_p\otimes g), (u_p\otimes g)\rangle (v_q\otimes h)\rangle \\
=& \langle (y_q\otimes h), \beta_g^{-1}(\langle x_p,u_p\rangle) (v_q\otimes h)\rangle \\
=& \beta_h^{-1}(\langle y_q, \beta_g^{-1}(\langle x_p,u_p\rangle) v_q\rangle).
\end{align*}
On the other hand,
\begin{align*}
&\langle M_{(p,g),(q,h)}((x_p\otimes g)\otimes (y_q\otimes h)), M_{(p,g),(q,h)}((u_p\otimes g)\otimes (v_q\otimes h))\rangle \\ 
=& 
\langle x_p\beta_g(y_q)\otimes g|_q h, u_p\beta_g(v_q)\otimes g|_p h\rangle \\
=& \beta_{h^{-1}}\beta_{(g|_p)^{-1}}(\langle x_p\beta_g(y_q), u_p\beta_g(v_q) \rangle) \\
=& \beta_{h^{-1}}\beta_{(g|_p)^{-1}}(\langle \beta_g(y_q), \langle x_p, u_p\rangle \beta_g(v_q) \rangle) \ (\text{as }M_{p,g\cdot q}\text{ is unitary}) \\
=&  \beta_{h^{-1}}\beta_{(g|_p)^{-1}}(\langle \beta_g(y_q), \beta_g(\beta_g^{-1}(\langle x_p, u_p\rangle) v_q) \rangle) \\
=& \beta_h^{-1}(\langle y_q, \beta_g^{-1}(\langle x_p,u_p\rangle) v_q\rangle )\ (\text{by (A6)}).
\end{align*}
Therefore, $M_{(p,g),(q,h)}$ preserves the inner product and is thus unitary. 

Finally, for $a\in X_p$, $b\in X_q$, $c\in X_r$ and $g,h,k\in G$, we compute that 
\[((a\otimes g)(b\otimes h))(c\otimes k) = (a\beta_g(b)\beta_{g|_q h}(c))\otimes (g|_q h)|_r k\]
and
\[(a\otimes g)((b\otimes h)(c\otimes k)) = (a\beta_g(b\beta_h(c)))\otimes g|_{q(h\cdot r)} h|_r k.\]
From \condref{A5} one has
\[a\beta_g(b\beta_h(c)) = a\beta_g(b) \beta_{g|_q}(\beta_h(c)) = a\beta_g(b)\beta_{g|_q h}(c).\]
Also \condref{ZS8} and \condref{ZS6} yield
\[(g|_q h)|_r k = (g|_q)|_{h\cdot r} h|_r k = g|_{q(h\cdot r)} h|_r k.\]
Hence the multiplication is associative. Therefore $Y=(Y_{(p,g)})_{(p,g)\in P\bowtie G}$ is a product system over the \ZS semigroup $P\bowtie G$.
\end{proof} 

\begin{defn}
The new product system $Y$ constructed from $(X,G,\beta)$ in Theorem \ref{T:ZSP} is called the \textit{\ZS product of $X$ by $G$} and denoted as $X\bowtie G$. 
\end{defn}

\subsection{Another Zappa-Sz\'{e}p product $X\widetilde\bowtie G$ of a homogeneous \ZS system}

In this subsection, we study a special class of \ZS actions of groups on product systems -- homogeneous \ZS actions. Given such a \ZS action, it turns out that 
\textit{another} new natural and interesting product system $X\widetilde\bowtie G$ can be constructed from the given action. Unlike $X\bowtie G$ that enlarges the grading semigroup and keeps the coefficient C*-algebra the same,
 this new one, $X\widetilde\bowtie G$, enlarges the coefficient C*-algebra and keeps 
the grading semigroup the same. 

\label{S:homog}
\begin{definition} 
Let $P\bowtie G$ be a \ZS product of a semigroups $P$ and a group $G$, and $X$ a product system over $P$.  
A \ZS action $\beta$ on $X$ is called \textit{homogeneous} if $g\cdot p=p$ for any $p\in P$ and $g\in G$. 
In this case,  the \ZS system $(X,G,\beta)$ is also said to be \textit{homogeneous}. 
\end{definition}

\begin{rem}
Homogeneous \ZS actions are a natural generalization of usual generalized gauge actions \cite{K2017}. 
\end{rem}

In the case of a homogeneous \ZS system $(X,G,\beta)$, one can easily see that $\beta$
induces an automorphic action $\beta:G\to \Aut(X_p)$. This allows us to construct a new crossed product type product system over the same semigroup $P$ that encodes the \ZS structure.
In particular, when $p= e$, we obtain a C*-dynamical system $(\cA,G,\beta)$. Let $\fA=\cA\rtimes_\beta G$ be the universal C*-crossed product of this C*-dynamical system. So $\fA=C^*(a, u_g: a\in \cA, g\in G)$, where $\{a,u_g: a\in \fA, g\in G\}$ is the generator set of $\fA$. 
Thus the generators satisfy the covariance condition
\[u_g a = \beta_g(a) u_g\qforal a\in \cA\text{ and }g\in G.\]

For each $p\in P$, consider $Z_p^0=c_{00}(G,X_p)=\spn\{x\otimes g: x\in X_p\}$. We can put an $\fA$-bimodule structure on $Z_p^0$: for any $a u_h\in \fA$ and $\xi=x_p\otimes g\in c_{00}(G,X_p)$,
\begin{align*}
    (au_h)\xi = (a\beta_h(x_p))\otimes h|_p g \quad \text{ and }\quad 
    \xi (au_h) = (x_p\beta_g(a)) \otimes gh.
\end{align*}
Define an $\fA$-valued function $\langle\cdot,\cdot\rangle: Z_p^0\times Z_p^0\to \fA$ by setting, for any $x_p\otimes g, y_p\otimes h\in Z_p^0$,
\[\langle x_p \otimes g, y_p \otimes h\rangle=\beta_{g^{-1}}(\langle x_p, y_p\rangle) u_{g^{-1}h}.\]
By the covariance relation on $\fA$, one can rewrite the above identity as 
\[\langle x_p \otimes g, y_p \otimes h\rangle=u_{g^{-1}} \langle x_p, y_p\rangle u_{h}.\]

\begin{proposition} The space $Z_p^0$ together with the map $\langle\cdot,\cdot\rangle$ is an inner product right $\fA$-module. 
\end{proposition}

\begin{proof}
It is easy to see that $\langle\cdot,\cdot\rangle$ is right $\fA$-linear in the second variable:
take any $au_k\in \fA$ and $x_p\otimes g, y_p\otimes h\in Z_p^0$,
\begin{align*}
    \langle x_p \otimes g, (y_p \otimes h) a u_k)\rangle &=  \langle x_p \otimes g, y_p\beta_h(a)\otimes hk\rangle \\
    &= \beta_{g^{-1}}(\langle x_p, y_p\beta_h(a)\rangle) u_{g^{-1} hk} \\
    &= \beta_{g^{-1}}(\langle x_p, y_p\rangle) \beta_{g^{-1}h}(a)  u_{g^{-1} h} u_k \\
    &= \beta_{g^{-1}}(\langle x_p, y_p\rangle) u_{g^{-1} h} a u_k \\
    &=  \langle x_p \otimes g, y_p \otimes h\rangle a u_k
\end{align*}
Also, for any $x_p\otimes g, y_p\otimes h\in Z_p^0$,
\begin{align*}
     \langle x_p \otimes g, y_p \otimes h\rangle^* &= \left(u_g^* \langle x_p, y_p\rangle u_h\right)^* \\
     &= u_h^* \langle y_p, x_p\rangle u_g \\
     &= \langle y_p \otimes h, x_p \otimes g\rangle.
\end{align*}
Finally, for any $x_1,\ldots,x_n\in X_p$ and $g_1,\ldots,g_n\in G$, considier $\xi=\sum_{i=1}^n x_i\otimes g_i\in Z_p^0$. We have that
\[\langle \xi, \xi\rangle = \sum_{i=1}^n \sum_{j=1}^n u_{g_i^{-1}} \langle x_i, x_j\rangle u_{g_j}. \]
Consider the $n\times n$ operator matrix $K=[\langle x_i, x_j\rangle]$. We first claim that $A\geq 0$ as an operator in $M_n(\cA)$, which is equivalent of showing \cite[Proposition 6.1]{Paschke1973} that for any $a_1,\ldots, a_n\in\cA$, 
\[\sum_{i=1}^n \sum_{j=1}^n a_i^* \langle x_i, x_j\rangle a_j \geq 0.\]
Since $X_p$ is an $\cA$-correspondence, $a_i^* \langle x_i, x_j\rangle a_j=\langle x_i a_i, x_j a_j\rangle$. Therefore,
\[\sum_{i=1}^n \sum_{j=1}^n a_i^* \langle x_i, x_j\rangle a_j=\sum_{i=1}^n \sum_{j=1}^n\langle x_i a_i, x_j a_j\rangle=\langle \sum_{i=1}^n x_i a_i, \sum_{j=1}^n x_j a_j\rangle\geq 0.\]
This proves that the operator matrix $K=[\langle x_i, x_j\rangle]\geq 0$. Since $\cA$ embeds injectively inside the crossed product $\fA=\cA\rtimes_\beta G$, the operator matrix $K\geq 0$ as an operator in $M_n(\fA)$. Therefore,
\[\langle \xi, \xi\rangle = \sum_{i=1}^n \sum_{j=1}^n u_{g_i^{-1}} \langle x_i, x_j\rangle u_{g_j}\geq 0. \]
Suppose that $\langle \xi,\xi\rangle=0$. We have 
\[\langle \xi, \xi\rangle = \sum_{i=1}^n \sum_{j=1}^n \beta_{g_i^{-1}}( \langle x_i, x_j\rangle) u_{g_i^{-1} g_j} = 0.\]
Since there exists a contractive conditional expectation $\Phi:\fA\to\cA$ by $\Phi(\sum a_g u_g)=a_e$, we have that 
\[\sum_{i=1}^n \beta_{g_i^{-1}}( \langle x_i, x_i\rangle)=0.\]
Since $\beta_{g_i^{-1}}$'s are $*$-automorphisms of $\cA$, we have that $\langle x_i, x_i\rangle =0$ for all $i$, and thus $x_i=0$ for all $i$. So we obtain that $\xi=0$. 
Therefore $\langle\cdot,\cdot\rangle$ is an $\fA$-valued inner product on $Z_p^0$. 
\end{proof}

Now let $Z_p$ be the completion of $Z_p^0$ under the norm $\|\xi\|=\|\langle \xi, \xi\rangle\|^{1/2}$. We obtain an $\fA$-correspondence $Z_p$. 

\begin{theorem}
\label{T:prop.Z} 
The collection $Z=\bigsqcup_{p\in P}Z_p$ is a product system over $P$, where the multiplication $Z_p\times Z_q\to Z_{pq}$ is given by 
\[(x_p\otimes g,y_q \otimes h)\mapsto x_p \beta_g(y_q)\otimes g|_q h \quad (g,h\in G, x_p\in X_p, y_q\in X_q).\] 
\end{theorem}

\begin{proof} Let $p, q\in P$. For $x_p\in X_p$ and $y_q\in X_q$, one has $\beta_g(y_q)\in X_q$ and thus $x_p \beta_g(y_q)\in X_{pq}$. Since $\beta_g$ is automorphic as mentioned above, 
the multiplication is surjective. To see the multiplication induces a unitary map from $Z_p\otimes Z_q\to Z_{pq}$, take any four elementary tensors $x_p\otimes g, w_p\otimes i\in Z_p$ and $y_q\otimes h, z_q\otimes k\in Z_q$,
\begin{align*}
    &\langle (x_p\otimes g)\otimes (y_q \otimes h), (w_p\otimes i)\otimes (z_q \otimes k)\rangle \\ 
    =& \langle y_q \otimes h, \langle x_p\otimes g, w_p\otimes i\rangle (z_q \otimes k)\rangle \\
    =&
    \langle y_q \otimes h, \beta_{g^{-1}}(\langle x_p, w_p\rangle) u_{g^{-1}i} (z_q \otimes k)\rangle \\
    =& \langle y_q \otimes h, \beta_{g^{-1}}(\langle x_p, w_p\rangle) \beta_{g^{-1}i}(z_q) \otimes (g^{-1}i)|_q k)\rangle \\
    =& u_h^* \langle y_q, \beta_{g^{-1}}(\langle x_p, w_p\rangle) \beta_{g^{-1}i}(z_q)\rangle u_{ (g^{-1}i)|_q k}.
\end{align*}
On the other hand,
\begin{align*}
    &\langle (x_p\otimes g) (y_q \otimes h), (w_p\otimes i) (z_q \otimes k)\rangle \\ 
    =& \langle x_p \beta_g(y_q)\otimes g|_q h, w_p\beta_i(z_q)\otimes i|_q k\rangle \\
    =& u_{g|_q h}^* \langle  x_p \beta_g(y_q),  w_p\beta_i(z_q)\rangle u_{i|_q k} \\
    =& u_h^* u_{(g|_q)^{-1}} \langle \beta_g(y_q), \langle x_p, w_p\rangle \beta_i(z_q)\rangle u_{i|_q k} \\
    =& u_h^* u_{(g|_q)^{-1}} \langle \beta_g(y_q), \beta_g(\beta_{g^{-1}}(\langle x_p, w_p\rangle) \beta_{g^{-1}i}(z_q))\rangle u_{i|_q k} \\
    =& u_h^* u_{(g|_q)^{-1}} \beta_{g|_q} (\langle y_q, \beta_{g^{-1}}(\langle x_p, w_p\rangle) \beta_{g^{-1}i}(z_q)\rangle) u_{i|_q k} \\
    =& u_h^* \langle y_q, \beta_{g^{-1}}(\langle x_p, w_p\rangle) \beta_{g^{-1}i}(z_q)\rangle u_{(g|_q)^{-1}}  u_{i|_q k} \\
    =& u_h^* \langle y_q, \beta_{g^{-1}}(\langle x_p, w_p\rangle) \beta_{g^{-1}i}(z_q)\rangle u_{g^{-1}|_{g\cdot q} i|_q k}.
\end{align*}
Because $g\cdot q=i\cdot q$, 
\[u_{g^{-1}|_{g\cdot q} i|_q k}=u_{g^{-1}|_{i\cdot q} i|_q k}=u_{(g^{-1}i)|_q k}.\]
Therefore, the multiplication is indeed unitary. 

The C*-algebra $Z_{e}$ can be identified as $\fA$ via $a\otimes g\mapsto au_g$. It is routine to verify that the left and right $\fA$-action on $Z_p$ are implemented by the multiplication map by $Z_{e}$. To see the associativity of the multiplication, take $x\otimes g\in Z_p$, $y\otimes h\in Z_q$, and $z\otimes k\in Z_r$. We have that
\[  \left((x\otimes g)(y\otimes h)\right)(z\otimes k) = x\beta_g(y)\beta_{g|_q h}(z)\otimes (g|_q h)|_r k\]
and
\[  (x\otimes g)\left((y\otimes h)(z\otimes k)\right) = x\beta_g(y\beta_h(z))\otimes g|_{qr} h|_r k.\]
Condition (A5) yields $\beta_g(y\beta_h(z))=\beta_g(y)\beta_{g|_q h}(z)$. From (ZS6), (ZS8), and the homogeneity assumption that $h\cdot r=r$, 
we have \[(g|_q h)|_r k=(g|_q)|_{h\cdot r} h|_r k=g|_{qr} h|_r k.\]
This proves the associativity of the multiplication. 
\end{proof}

\begin{definition} 
The product system $Z$ obtained in Theorem \ref{T:prop.Z} is called the \textit{homogeneous \ZS product of $X$ by $G$} and denoted as $X\widetilde\bowtie G$. 
\end{definition}

In summary, for a given \textsf{homogeneous} \ZS action $\beta$ of $G$ on $X$, one has two new product systems: (i) $X\bowtie G$ -- a product system over $P\bowtie G$ with coefficient C*-algebra $\cA$,
and (ii) $X\widetilde\bowtie G$ -- a product system over $P$ with coefficient C*-algebra $\cA\rtimes_\beta G$.

%%%%%%%%%%%%%%%%
%%%%%%%%%%%%%%%%

\section{C*-algebras associated to \ZS actions \\ and some Hao-Ng isomorphism theorems}

\label{S:main}

In this section, we study covariant representations and their associated universal C*-algebras arising from a \ZS system $(X,G,\beta)$. 
We establish a one-to-one correspondence between the set of all covariant representations $(\psi, U)$ of $(X,G,\beta)$ and the set of all (Toeplitz) representations $\Psi$
of the \ZS product system $X\bowtie G$. Furthermore, it is proved that $\psi$ is Cuntz-Pimsner covariant if and only if so is $\Psi$. If $P$ is right LCM and $X$ is compactly 
aligned, then $\psi$ is Nica-covariant if and only if so is $\Psi$. As a consequence, we obtain several Hao-Ng isomorphism theorems for the associated C*-algebras. 
However, as we shall see, changes arise for the homogeneous \ZS product system $X\widetilde\bowtie G$.

\subsection{Covariant representations of $(X,G,\beta)$ and the C*-algebra $\T_X\bowtie G$}

Let $P\bowtie G$ be a \ZS product of a semigroup $P$ and a group $G$, $\A$ a unital C*-algebra, and $X$ a product system over $P$ with coefficient $\A$. Suppose that $\beta$ is an arbitrary \ZS action of $G$ on $X$.

\begin{defn}
\label{D:vrep}
Let $\psi$ be a representation of $X$ and $U$ is unitary representation of $G$ on a unital C*-algebra $\cB$. A pair $(\psi, U)$ is called a \textit{covariant representation of a \ZS system $(X,G,\beta)$} if 
\begin{align}
\label{E:Upsi}
U_g \psi_p(x)=\psi_{g\cdot p}(\beta_g(x)) U_{g|_p} \foral p\in P, g\in G, x\in X_p. 
\end{align}
\end{defn}

\begin{theorem}
\label{T:Upsi}
There is a one-to-one correspondence $\Pi$ between the set of all representations $\Psi$ of $X\bowtie G$ on a unital C*-algebra $\cB$
and the set of covariant representations $(\psi, U)$ of $(X,G,\beta)$.

In fact, for a given representation $\Psi$ of $X\bowtie G$, one has 
\begin{align*}
\psi_p(x)&:=\Psi_{(p,e)}(x\otimes e)\qforal p\in P, x\in X_p,\\
U_g&:=\Psi_{(e, g)}(1_\cA\otimes g)\qforal g\in G.
\end{align*}
Conversely, given a covariant representation $(\psi, U)$ of $(X,G,\beta)$, one has 
\begin{align}
\label{E:DefPsi}
\Psi_{(p, g)}(x\otimes g) := \psi_p(x) U_g\qforal p\in P, g\in G, x\in X_p.
\end{align}
\end{theorem}

\begin{proof}  Let $Y:=X\bowtie G$ and $\Psi$ be a representation of $Y$ on a unital C*-algebra $\cB$. 
Since $\Psi$ is a representation of $Y$, in particular we have 
\begin{align*}
%\label{E:Psi}
\Psi_{(p,g)}(x\otimes g)\Psi_{(q,h)}(y\otimes h)=\Psi_{(p(g\cdot q), g|_qh)}(x\beta_g(y)\otimes g|_qh)
\end{align*}
for all $g, h\in G$, $p,q\in P$, $x\in X_p$ and $y\in X_q$. 

For $p\in P$, define $\psi_P: X_p \to \cB$ via
\begin{align*}
%\label{E:psi}
\psi_p(x):=\Psi_{(p,e)}(x\otimes e)\qforal x\in X_p.
\end{align*}
Then 
\begin{align*}
\psi_p(x)^*\psi_p(y)
&=\Psi_{(p,e)}(x\otimes e)^*\Psi_{(p\otimes e)}(y\otimes e)\\
 &=\Psi_{(e,e)}(\langle (x\otimes e), (y\otimes e)\rangle\\
 &=\Psi_{(e,e)}(\langle x, y\rangle\otimes  e)\\
 &=\psi_{e}(\langle x, y\rangle)\ (x, y\in X_p)
\end{align*}
and 
\begin{align*}
\psi_p(x)\psi_q(y)
&=\Psi_{(p,e)}(x\otimes e)\Psi_{(q,e)}(y\otimes e)\\
&=\Psi_{(p(e\cdot q), e|_q e)}((x\otimes e)(y\otimes e))\\
&=\Psi(pq,e)(x\beta_{e}(y)\otimes e|_q e)\\
&=\Psi(pq,e)(xy\otimes e)\\
&=\psi_{pq}(xy)\ (x\in X_p, y\in X_q). 
\end{align*}
So $\psi$ is a representation of $X$ on $\cB$. 

One can easily check that $U_g:=\Psi_{(e, g)}(1_\cA\otimes g)$ is a unitary in $\cB$ with inverse $\Psi_{(e, g^{-1})}(1_\cA\otimes g^{-1})$, and that 
\[U_g U_h=\Psi_{(e, g)}(1_\cA\otimes g) \Psi_{(e, h)}(1_\cA\otimes h)= \Psi_{(e, gh)}(1_\cA\otimes gh)=U_{gh}.\] That is, 
$U$ is a unitary representation of $G$ in $\cB$. 

From the definitions of $\psi_p$ and $U_g$, for any $g\in G$ and $p\in P$, we obtain
\begin{align*}
U_g \psi_p(x) &= \Psi_{(e,g)}(1_\cA\otimes g) \Psi_{(p,e)}(x\otimes e) \\
&= \Psi_{(e,g)(p,e)}((1_\cA\otimes g)(x\otimes e)) \\
&= \Psi_{(g\cdot p,g|_p)}(\beta_g(x)\otimes g|_p) \\
&= \Psi_{(g\cdot p,e)(e,g|_p)}((\beta_g(x)\otimes e)(1_\cA\otimes g|_p)) \\ 
&= \Psi_{(g\cdot p,e)}(\beta_g(x)\otimes e) \Psi_{(e,g|_p)}(1_\cA\otimes g|_p) \\
&= \psi_{g\cdot p}(\beta_g(x)) U_{g|_p}.
\end{align*}
Thus $(\psi, U)$ satisfies the idenity \eqref{E:Upsi}. 

Conversely, given a representation $\psi$ of $X$ and a unitary representation $U$ of $G$ on a unital C*-algebra $\B$ that satisfy \eqref{E:Upsi},
define 
\[
\Psi_{(p, g)}(x\otimes g) := \psi_p(x) U_g\qforal p\in P, g\in G, x\in X_p.
\]
Then one can verify that $\Psi$ is a representation of $Y$. In fact, we have 
\begin{align*}
\Psi_{(p, g)}(x\otimes g) ^*\Psi_{(p, g)}(y\otimes g) 
&=  U_g^*\psi_p(x)^*  \psi_p(y) U_g\ (\text{by the definition of } \Psi)\\
&= U_{g^{-1}} \psi_{e}(\langle x, y\rangle) U_g \ (\text{as }\psi \text{ is a representation of } X)\\
&=\psi_{e}(\beta_{g^{-1}}\langle x, y\rangle) U_{g^{-1}|_{e}} U_g \\
&=\psi_{e}(\beta_{g^{-1}}\langle x, y\rangle)\ (\text{as }g^{-1}|_{e}=g^{-1}, U_{g^{-1}|_{e}} U_g=U_{g^{-1}}U_g=I) \\
&=\Psi_{(e,e)}(\beta_{g^{-1}}(\langle x, y\rangle)\otimes e) \\
&=\Psi_{(e,e)}(\langle x\otimes g, y\otimes g\rangle)\ (\text{by the definition of }Y)
\end{align*}
for all $p\in P$, $g\in G$, $x,y\in X_p$, and 
\begin{align*}
\Psi_{(p, g)}(x\otimes g) \Psi_{(q, h)}(y\otimes h) 
&=\psi_p(x) U_g\psi_q(y) U_h \ (\text{by definition of } \Psi)\\
&=\psi_p(x) \psi_{g\cdot q}(\beta_g(y)) U_{g|_q} U_h\\
&=\psi_{pg\cdot q}(x\beta_g(y)) U_{g|_q h}\ (\text{as }\psi \text{ is a representation of } X)\\
&=\Psi_{(pg\cdot q, g|_qh)}(x\beta_g(y)\otimes g|_q h)\ (\text{by the definition of } \Psi)\\
&=\Psi_{(pg\cdot q, g|_qh)}((x\otimes g)(y\otimes h))
\end{align*}
for all $p,q\in P$, $g\in G$, $x\in X_p$ and $y\in X_q$.
\end{proof}

\begin{eg}
\label{Eg:Fock}
Let $X$ be a product system over a semigroup $P$. Suppose that there is a \ZS action $\beta$ of $G$ on $X$. 
There is a natural nontrivial pair $(\psi, U)$ which satisfies all the conditions required in Theorem \ref{T:Upsi}.  
Indeed, let $F(X)=\bigoplus_{s\in P} X_s$ the Fock space, and $L$ the usual Fock representation:
\begin{align*}
L_p(x)(\oplus x_s) =\oplus (x\otimes x_s)\quad (x\in X_p, \ \oplus x_s\in F(X)).
\end{align*}
Then we define an action $\tilde\beta$ of $G$ on $F(X)$ as follows: 
\[
\tilde\beta_g(\oplus x_s)=\oplus (\beta_g(x_s))\quad (g\in G, \ \oplus x_s\in F(X)). 
\]
Clearly, $L$ is a representation of $X$ and $\tilde\beta$ is a unitary representation of $G$. Also one can easily verify 
\begin{align*}
\tilde\beta_g\circ L_p(x)=L_{g\cdot p}(\beta_g(x))\circ \tilde\beta_{g|_p}\quad (g\in G, x\in X_p).
\end{align*}
\end{eg}

We are now ready to define the Toeplitz type universal C*-algebra associated to a \ZS system $(X,G,\beta)$. 

\begin{definition} 
\label{D:TXG}
Let $\cT_X\bowtie G$ be the universal C*-algebra generated by covariant representations of $(X,G,\beta)$. 
\end{definition}

Example \ref{Eg:Fock} shows that the C*-algebra $\cT_X\bowtie G$ is nontrivial. 

\medskip
In what follows, we prove a result similar to Theorem \ref{T:Upsi} for the homogeneous \ZS product system $X\widetilde\bowtie G$.

\begin{theorem}
\label{T:Upsi.homo} 
Suppose that a \ZS system $(X,G,\beta)$ is homogeneous.  
Then there is a one-to-one correspondence $\widetilde\Pi$ between the set of all representations $\Psi$ of $X\widetilde\bowtie G$ on a unital C*-algebra $\cB$
and the set of all covariant representations $(\psi, U)$ of $(X,G,\beta)$.
In fact, for a given representation $\Psi$ of $X\widetilde\bowtie G$, one has 
\begin{align*}
\psi_p(x)&:=\Psi_{p}(x\otimes e)\qforal p\in P, x\in X_p,\\
U_g&:=\Psi_{e}(1_\cA\otimes g)\qforal g\in G.
\end{align*}
Conversely, given a covariant representation $(\psi, U)$ of $X\widetilde\bowtie G$, one has 
\[
\Psi_{p}(x\otimes g) := \psi_p(x) U_g\qforal p\in P, g\in G, x\in X_p.
\]
\end{theorem}

\begin{proof} 
For simplicity, set $Z:=X\widetilde\bowtie G$.
Let $\Psi$ be a representation of $Z$ on $\cB$. For all $p,q\in P$, $g,h\in G$, and $x\in X_p$, $y\in X_q$, 
\[\Psi_p(x\otimes g)\Psi_q(y\otimes h)=\Psi_{pq}(x\beta_g(y)\otimes g|_q h).\]
 For $p\in P$, define $\psi_p:X_p\to \cB$ by
 \[\psi_p(x):=\Psi_p(x\otimes e)\ \text{for all }x\in X_p.\]
Here, for $p=e$, we treat $x\otimes e=xu_{e}\in \cA\rtimes_\beta G\cong Z_{e}$.
 Then,
 \begin{align*}
     \psi_p(x)^* \psi_p(y) &= \Psi_p(x\otimes e)^* \Psi_p(y\otimes e) \\
     &= \Psi_{e}(\langle x\otimes e, y\otimes e\rangle) \\
     &= \Psi_{e}(\langle x,y\rangle u_{e})\\
     &=\psi_{e}(\langle x,y\rangle) \\
     \psi_p(x)\psi_q(y) &= \Psi_p(x\otimes e) \Psi_q(y\otimes e) \\\
     &= \Psi_{pq}(x\beta_{e}(y)\otimes e|_q e) \\
     &= \Psi_{pq}(xy\otimes e) \\
     &=\psi_{pq}(xy).
 \end{align*}
 Therefore, $\psi$ is a representation of $X$ on $\cB$.
 
For each $g\in G$, set $U_g:=\Psi_{e}(1_\cA\otimes g)$. As before, one can easily check that $U$ is a unitary representation of $G$ on $\cB$. 
 
 For any $g\in G$, $p\in P$ and $x\in X_p$,
 \begin{align*}
     U_g \psi_p(x) &= \Psi_{e}(1_\cA\otimes g) \Psi_p(x\otimes e) \\
     &= \Psi_p(\beta_g(x)\otimes g|_p) \\
     &= \Psi_p((\beta_g(x)\otimes e)(1_\cA\otimes g|_p)) \\
     &= \Psi_p(\beta_g(x)\otimes e)\Psi_{e}(1_\cA\otimes g|_p) \\
     &= \psi_p(\beta_g(x)) U_{g|_p}.
 \end{align*}
 
 Conversely, given a representation $\psi$ of $X$ and a unitary representation $U$ of $G$ that satisfy $U_g\psi_p(x)=\psi_p(\beta_g(x)) U_{g|_p}$, define
 \[
 \Psi_p(x\otimes g):=\psi_p(x) U_g\qforal g\in G, p\in P, x\in X_p.
 \]
 For any $x,y\in X_p$ and $g,h\in G$, 
 \begin{align*}
 \Psi_p(x\otimes g)^* \Psi_p(y\otimes h) &= U_g^* \psi_p^*(x) \psi_p(y) U_h \\
 &= U_g^* \psi_{e}(\langle x,y\rangle) U_h \\
 &= \psi_{e}(\beta_{g^{-1}}(\langle x,y\rangle)) U_{g^{-1}h} \\
 &= \Psi_{e}(\beta_{g^{-1}}(\langle x,y\rangle)\otimes g^{-1} h) \\
 &= \Psi_{e}(\langle x\otimes g, y\otimes h\rangle).
 \end{align*}
 For any $x\in X_p$, $y\in X_q$, and $g,h\in G$,
 \begin{align*}
      \Psi_p(x\otimes g) \Psi_q(y\otimes h) &= \psi_p(x) U_g \psi_q(y) U_h \\
      &= \psi_p(x) \psi_q(\beta_g(y)) U_{g|_q} U_h \\
      &= \psi_{pq}(x\beta_g(y)) U_{g|_q h} \\
      &= \Psi_{pq}(x\beta_g(y)\otimes g|_q h) \\
      &= \Psi_{pq}((x\otimes g)(y\otimes h)). 
 \end{align*}
 Therefore, $\Psi$ is a representation of $Z$. 
\end{proof}

As an immediate corollary of Theorems \ref{T:Upsi} and \ref{T:Upsi.homo}, the universal $C^*$-algebras of representations of $X\bowtie G$ 
and covariant representations of $(X,G,\beta)$ must coincide, and similar for $X\widetilde\bowtie G$ if $(X,G,\beta)$ is also homogeneous. 
Therefore, we have the following Toeplitz type Hao-Ng isomorphism theorem. 
   
\begin{corollary} 
\label{C:HaoNgT}
Let $(X,G,\beta)$ be a \ZS system. Then 
\begin{itemize}
\item[(i)] $\cT_{X\bowtie G} \cong \cT_X \bowtie G$; and 

\item[(ii)] 
$\cT_{X\bowtie G} \cong \cT_X \bowtie G \cong \cT_{X\widetilde\bowtie G}$ provided that $(X,G,\beta)$ is homogeneous. 
\end{itemize}
\end{corollary}

\subsection{Cuntz-Pimsner type representations of $(X,G,\beta)$ and the C*-algebra $\cO_X\bowtie G$}

In this subsection, we prove that the one-to-one correspondence $\Pi$ in Theorem \ref{T:Upsi} preserves the Cuntz-Pimsner covariance: $\Psi$ is Cuntz-Pimsner covariant
if and only if so is $\psi$. Accordingly we obtain the Hao-Ng isomorphism theorem as well in this case. But unfortunately, it is unknown whether 
the correspondence $\widetilde\Pi$ in Theorem \ref{T:Upsi.homo} preserves the Cuntz-Pimsner covariance.

\begin{proposition}\label{prop.cp1}
Let $(X,G,\beta)$ be a \ZS system and set $Y:=X\bowtie G$. 
For each $p\in P$ and $g\in G$, define a map $\iota_{p,g}:\cL(X_p)\to\cL(Y_{(p,g)})$ by
\[\iota_{p,g}(T)(x\otimes g)=(Tx)\otimes g.\]
Then $\iota_{p,g}$ is an isometric $*$-isomorphism. Moreover, for each rank-one operator $\theta_{x,y}\in \cK(X_p)$, $\iota_{p,g}(\theta_{x,y})$ is the rank one operator $\Theta_{x\otimes g, y\otimes g}\in\cK(Y_{(p,g)})$, and thus $\iota_{p,g}(\cK(X_p))=\cK(Y_{(p,g)})$.
\end{proposition}

\begin{proof} For any $T\in\cL(X_p)$ and $x\in X_p$, $g\in G$,
\begin{align*}
    \|\iota_{p,g}(T)(x\otimes g)\|^2 &= \|\langle (Tx)\otimes g, (Tx)\otimes g\rangle\| \\
    &= \|\beta_{g^{-1}}(\langle Tx, Tx\rangle\| \\
    &= \|\langle Tx, Tx\rangle\|=\|Tx\|^2.
\end{align*}
Thus $\iota_{p,g}$ is isometric. It is clear that $\iota_{p,g}$ is a homomorphism. Moreover, for any $x,y\in X_p$,
\begin{align*}
    \langle \iota_{p,g}(T^*)x\otimes g, y\otimes g\rangle &= \langle (T^*x)\otimes g, y\otimes g\rangle \\
    &= \beta_{g^{-1}}(\langle T^* x,y\rangle)= \beta_{g^{-1}}(\langle x, Ty\rangle) \\
    &= \langle x\otimes g, (Ty)\otimes g\rangle = \langle x\otimes g, \iota_{p,g}(T)y\otimes g\rangle.
\end{align*}
Hence $\iota_{p,g}(T^*)=\iota_{p,g}(T)^*$. Finally, for any $\widetilde{T}\in\cL(Y_{(p,g)})$, take any $x\in X_p$ and define $Tx\in X_p$ such that $\widetilde{T}(x\otimes g)=Tx\otimes g$. One can check that $T$ is a $\cA$-linear, adjointable operator in $\cL(X_p)$, and that $\widetilde{T}=\iota_{p,g}(T)$. Therefore, $\iota_{p,g}$ is an isometric $*$-isomorphism.

Now fix a rank one operator $\theta_{x,y}\in\cK(X_p)$. In other words, for any $z\in X_p$, $\theta_{x,y}(z)=x\langle y,z\rangle$. Now for any $g\in G$ and $z\in X_p$, we have
\begin{align*}
\iota_{p,g}(\theta_{x,y})(z\otimes g) 
    &= \theta_{x,y}(z)\otimes g 
      = (x\langle y,z\rangle)\otimes g \\
    &= (x\otimes g)(\beta_{g^{-1}}(\langle y,z\rangle)\otimes e) \\
    &= (x\otimes g)\langle y\otimes g, z\otimes g\rangle \\
    &= \Theta_{x\otimes g, y\otimes g}(z\otimes g). 
\end{align*}
This proves $\iota_{p,g}(\theta_{x,y})=\Theta_{x\otimes g, y\otimes g}$, and therefore $\iota_{p,g}(\cK(X_p))=\cK(Y_{(p,g)})$
\end{proof}

Let $\phi_p$ and $\Phi_{(p,q)}$ be the left action of $\cA$ on $X_p$ and $Y_{(p,g)}$, respectively: 
\[
\phi_p(a)x=ax\text{ and } \Phi_{(p,g)}(a)(x\otimes g)=ax\otimes g.
\]

\begin{theorem}\label{thm.cp} 
Let $\Psi$ be a representation of $X\bowtie G$ and $(\psi, U)$ be the covariant representation of $(X,G,\beta)$ under the one-to-one correspondence $\Pi$ given in Theorem \ref{T:Upsi}. Then $\Psi$ is Cuntz-Pimsner covariant if and only if so is $\psi$. 
\end{theorem}

\begin{proof} 
Suppose yhat $\Psi$ is a Cuntz-Pimsner covariant representation of $Y$. Then by definition, for any $(p,g)\in P\bowtie G$, 
\[\Psi^{(p,g)}(\Phi_{(p,g)}(a))=\Psi_{(e,e)}(a)\ \text{ for all }a\in\Phi_{(p,g)}^{-1}(\cK(Y_{(p,g)})).\]
Notice that, for any $p\in P$ and $a\in \A$ with $\phi_p(a)\in \cK(X_p)$, we have $\Phi_{(p,e)}(a)\in \cK(Y_{(p,e)})$. Thus
\[\psi^{(p)}(\phi(a))=\Psi^{(p,e)}(\Phi(a))=\Psi_{(e,e)}(a)=\psi_{e}(a).\]
Therefore, $\psi$ is Cuntz-Pimser covariant. 

Conversely, suppose that $\psi$ is Cuntz-Pimsner covariant. Take $a\in \Phi^{-1}(\cK(Y_{(p,g)})$. Without loss of generality, we assume that 
\[\Phi_{(p,g)}(a)=\Theta_{x\otimes g, y\otimes g}\ \text{for }x,y\in X_p.\]
By Proposition \ref{prop.cp1}, \[\Phi_{(p,g)}(a)=\Theta_{x\otimes g, y\otimes g}=\iota_{p,g}(\theta_{x,y})=\iota_{p,g}(\phi_p(a)).\]
We have that $\theta_{x,y}=\phi_p(a)$, and so 
\begin{align*}
    \Psi^{(p,g)}(\Phi_{(p,g)}(a)) &= \Psi^{(p,g)}(\Theta_{x\otimes g, y\otimes g}) \\
    &= \Psi_{(p,g)}(x\otimes g)\Psi_{(p,g)}(y\otimes g)^* \\
    &= \psi_p(x) U_g (\psi_p(y) U_g)^* = \psi_p(x) U_g U_g^* \psi_p(y)^* \\
    &= \psi_p(x) \psi_p(y)^* = \psi^{(p)}(\theta_{x,y}) \\
    &= \psi^{(p)}(\phi_p(a)) 
      = \psi_{e}(a)=\Psi_{(e,e)}(a).
\end{align*}
Therefore, $\Psi$ is Cuntz-Pimsner covariant. 
\end{proof}

\begin{definition} 
\label{D:OXG}
Let $\cO_X\bowtie G$ be the universal C*-algebra generated by the set of all covariant representations $(\psi,U)$ of $(X,G,\beta)$ with $\psi$ Cuntz-Pimsner covariant. 
\end{definition}

As a corollary of Theorem \ref{thm.cp}, the universal C*-algebra of Cuntz-Pimsner covariant representations of $X\bowtie G$ and the universal C*-algebra of covariant representations $(\psi,U)$ 
with $\psi$ Cuntz-Pimsner covariant of $(X,G,\beta)$ must coincide. Therefore, we have the following Cuntz-Pimsner type Hao-Ng isomorphism theorem. 

\begin{corollary} 
\label{C:HaoNgC}
$\cO_{X\bowtie G} \cong \cO_X\bowtie G$.
\end{corollary}

%\begin{proof} For any representation $(\pi, U)$ of $\cO_X\bowtie G$, 
%%$\psi=\pi\circ j_X$ is a Cuntz-Pimsner covariant representation of $X$, and thus 
%by Theorem \ref{thm.cp} there exists a Cuntz-Pimsner covariant representation $\Psi$ of $X\bowtie G$ such that $\Psi(x\otimes g)=\psi(x)U_g$ for all $x\in X$ and $g\in G$. 
%
%Conversely, every Cuntz-Pimsner covariant representation $\Psi$ of $X\bowtie G$ corresponds to a Cuntz-Pimsner covariant representation $\psi$ of $X$ and a unitary representation $U$ of $G$ such that $U_g \psi(x)=\psi(\beta_g(x)) U_{g|_p}$ for all $p\in P$ and $x\in X_p$. By the universality of $j_X$, there exists a $*$-homomorphism $\psi_*$ of $\cO_X$ such that $\psi=\psi_* \circ j_X$. The pair $(\psi_*, U)$ is a representation for $\cO_X\bowtie G$. 
%
%By the universality of both C*-algebras, we have that $\cO_{X\bowtie G}\cong \cO_X\bowtie G$.
%\end{proof}

\begin{remark} 
One might notice that the above corollary has no corresponding part to Corollary \ref{C:HaoNgT} (ii) for the homogeneous case. 
In fact, we do not know whether $\cO_{X\widetilde\bowtie G}\cong \cO_X\bowtie G$, although $\cO_{X\widetilde\bowtie G}$ is generally a quotient of $\cO_{X\bowtie G}$ (see Corollary \ref{C:2bowtie} below).
Even in the special situation of a homogeneous \ZS action of $G$ on $X$ where $g|_p=g$ for all $g\in G$ and $p\in P$,
the \ZS homogeneous product system $X\widetilde\bowtie G$ becomes a crossed product $X\rtimes G$. In such a case, the problem whether  $\cO_{X\rtimes G}\cong \cO_X\rtimes G$ is known as the Hao-Ng isomorphism problem in the literature. The isomorphism is known in several special cases (see, for example, \cite{HaoNg2008} and more recent approaches from non-self-adjoint operator algebras \cite{DOK20, K2017, KR19}).
\end{remark}

\begin{corollary}
\label{C:2bowtie}
If a \ZS system $(X,G,\beta)$ is homogeneous, then there is a natural epimorphism from $\cO_{X\bowtie G}$ to  $\cO_{X\widetilde\bowtie G}$.
\end{corollary}

\begin{proof}
By Corollary \ref{C:HaoNgT}, there is a natural Cuntz-Pimsner covariant representation of $X\bowtie G$ on $\O_{X\widetilde\bowtie G}$:
$X\bowtie G\hookrightarrow \cT_{X\bowtie G}\cong \cT_{X\widetilde\bowtie G} \twoheadrightarrow \O_{X\widetilde\bowtie G}$. By the universal property of $\O_{X\bowtie G}$, this gives a homomorphism from $\cO_{X\bowtie G}$ onto $\cO_{X\widetilde\bowtie G}$. 
\end{proof}

\subsection{Nica-Toeplitz representations of $X\bowtie G$ and the C*-algebra $\cN\cT_X\bowtie G$}

Suppose that $P$ is a right LCM semigroup and that $G$ is a group. Then the \ZS product semigroup $P\bowtie G$ is known to be right LCM as well \cite{BRRW}. In particular, for any $(p,g), (q,h)\in P\bowtie G$, we have that 
\[(p,g)P\bowtie G\cap (q,h)P\bowtie G = \begin{cases}
(r,k)P\bowtie G & \text{if } pP\cap qP=rP, k\in G,\\
\emptyset & \text{otherwise}.
\end{cases}
\]
Here, the choice of $k$ can be arbitrary since $(e,k)$ is invertible in $P\bowtie G$ and thus $(r,k)P\bowtie G=(r,e)P\bowtie G$ for all $k\in G$.
 We first prove that, for a given \ZS system $(X,G,\beta)$ where $X$ is compactly aligned, the product system $X\bowtie G$ is compactly aligned as well.

\begin{proposition} 
Let $P$ be a right LCM semigroup and $X$ a compactly aligned product system over $P$. Then, for a given \ZS system $(X,G,\beta)$, $X\bowtie G$ is also compactly aligned. 
\end{proposition}

\begin{proof} As before, let $Y:=X\bowtie G$. For each $p\in P$ and $g\in G$, define $\iota_{p,g}:\cL(X_p)\to\cL(Y_{(p,g)})$ by $\iota_{p,g}(T)(x\otimes g)=(Tx)\otimes g$. By Proposition \ref{prop.cp1}, $\iota_{p,g}$ is an isometric $*$-isomorphism, and $\iota_{p,g}(\cK(X_p))=\cK(Y_{(p,g)})$. Fix $p,q\in P$ with $pP\cap qP=rP$. For any $g,h\in G$, we have that $(p,g)P\bowtie G\cap (q,h)P\bowtie G=(r,e)P\bowtie G$. 
Let $p^{-1}r$ be the unique element in $P$ such that $p(p^{-1}r)=p$. 

For any $S\in\cK(X_p)$ and any $x\in X_p$ and $y\in X_{p^{-1}r}$, one has 
\begin{align*}
    \iota_{r,e}\circ i_p^r(S)(xy\otimes e) 
    &= (i_p^r(S)(xy))\otimes e\\
    &= (Sx)y \otimes e \\
    &= ((Sx)\otimes g)(\beta_{g^{-1}}(y)\otimes g_0)\  (\text{with }g_0:=(g|_{g^{-1}\cdot (p^{-1}r)})^{-1})\\
    &= \left(\iota_{p,g}(S)(x\otimes g)\right)(\beta_{g^{-1}}(y)\otimes g_0)\\
    &= i_{(p,g)}^{(r,e)}\circ \iota_{p,g}(S)(xy\otimes e).
\end{align*}
%Here, $g_0:=(g|_{g^{-1}\cdot (p^{-1}r)})^{-1}$. 
Similarly, we have that $\iota_{r,e}\circ i_q^r=i_{(q,h)}^{(r,e)}\circ \iota_{q,h}$. Therefore, the following diagram commutes:

\begin{figure}[h]
    \centering

    \begin{tikzpicture}[scale=0.9]
    
    \node at (-3,2) {$\cK(X_p)$};
    \node at (3,2) {$\cK(Y_{(p,g)})$};
    \node at (-3,0) {$\cK(X_r)$};
    \node at (3,0) {$\cK(Y_{(r,e)})$};
    \node at (-3,-2) {$\cK(X_q)$};
    \node at (3,-2) {$\cK(Y_{(q,h)})$};
    
    \draw[->] (-2,2) -- (2,2);
    \draw[->] (-2,0) -- (2,0);
    \draw[->] (-2,-2) -- (2,-2);
    \draw[->] (-3,1.5) -- (-3,0.5);
    \draw[->] (-3,-1.5) -- (-3,-0.5);
    \draw[->] (3,1.5) -- (3,0.5);
    \draw[->] (3,-1.5) -- (3,-0.5);

	\node at (-3.3, 1) {$i_p^r$};
	\node at (-3.3, -1) {$i_q^r$};
	
	\node at (0, 2.3) {$\iota_{p,g}$};
    \node at (0, 0.3) {$\iota_{r,e}$};
    \node at (0, -1.7) {$\iota_{q,h}$};
	
	\node at (3.6, 1) {$i_{(p,g)}^{(r,e)}$};
	\node at (3.6, -1) {$i_{(q,h)}^{(r,e)}$};
    \end{tikzpicture}
\end{figure}

Now for any $S\in\cK(Y_{(p,g)})$ and $T\in\cK(Y_{(q,h)})$, by Propostion \ref{prop.cp1}, there exist $S_0\in\cK(X_p)$ and $T_0\in\cK(X_q)$ such that $\iota_{p,g}(S_0)=S$ and $\iota_{q,h}(T_0)=T$. Therefore,
\begin{align*}
    i_{(p,g)}^{(r,e)}(S)i_{(q,h)}^{(r,e)}(T) &= i_{(p,g)}^{(r,e)}(\iota_{p,g}(S_0))i_{(q,h)}^{(r,e)}(\iota_{q,h}(T_0)) \\
    &= \iota_{r,e}(i_p^r(S_0))\iota_{r,e}(i_q^r(T_0)) \\
    &= \iota_{r,e}(i_p^r(S_0)i_q^r(T_0)).
\end{align*}
Since $X$ is compactly aligned, $i_p^r(S_0)i_q^r(T_0)\in \cK(X_r)$. Since $\iota_{r,e}$ maps $\cK(X_r)$ to $\cK(Y_{(r,e)})$, we have that
\[
i_{(p,g)}^{(r,e)}(S)i_{(q,h)}^{(r,e)}(T)=\iota_{r,e}(i_p^r(S_0)i_q^r(T_0))\in\cK(Y_{(r,e)}). 
\qedhere\]
\end{proof}

\begin{theorem} 
\label{T:Ncov}
Suppose that $X$ is a compactly aligned product system over a right LCM semigroup $P$. 
Let $(X,G,\beta)$ be a \ZS system, $\Psi$ a representation of $X\bowtie G$, and $(\psi, U)$ the covariant representation of $(X,G,\beta)$ under the one-to-one 
correspondence $\Pi$ given in Theorem \ref{T:Upsi}.
Then $\Psi$ is Nica-covariant if and only if so is $\psi$. 
\end{theorem}

\begin{proof}
For any $p\in P$ and $g\in G$, we first show that $\Psi^{(p,g)}\circ \iota_{p,g}=\psi^{(p)}$. For any $\theta_{x,y}\in \cK(X_p)$, $\psi^{(p)}(\theta_{x,y})=\psi_p(x) \psi_p(y)^*$. By Proposition \ref{prop.cp1}, $\iota_{p,g}(\theta_{x,y})=\Theta_{x\otimes g, y\otimes g}$, and thus
\begin{align*}
\Psi^{(p,g)}(\iota_{p,g}(\theta_{x,y})) &= \Psi_{(p,g)}(x\otimes g) \Psi_{(p,g)}(y\otimes g)^* \\
&=\psi_p(x) U_g U_g^* \psi_p(y)^* \ (\text{by }\eqref{E:DefPsi})\\
&= \psi_p(x)\psi_p(y)^* = \psi^{(p)}(\theta_{x,y}).
\end{align*}
In other words, the following diagram commutes:
\begin{figure}[h]
    \centering
    \begin{tikzpicture}[scale=0.9]
    
    \node at (-3,2) {$\cK(X_p)$};
    \node at (2.5,2) {$\cB$};
    \node at (-3,0) {$\cK(Y_{p,g})$};
    
    \draw[->] (-3,1.5) -- (-3,0.5);
    \draw[->] (-2,2) -- (2.3,2);
    \draw[->] (-2,0) -- (2.3,1.8);

	\node at (0, 2.3) {$\psi^{(p)}$};
    \node at (0, 0.3) {$\Psi^{(p,g)}$};
    \node at (-3.5, 1) {$\iota_{p,g}$};
    \end{tikzpicture}
\end{figure}

Suppose that $\psi$ is Nica-covariant. For any $S\in \cK(Y_{(p,g)})$ and $T\in \cK(Y_{(q,h)})$, there exist $S_0\in\cK(X_p)$ and $T_0\in\cK(X_q)$ such that $\iota_{p,g}(S_0)=S$ and $\iota_{q,h}(T_0)=T$. 
Then we have
\[ \Psi^{(p,g)}(S)\Psi^{(q,h)}(T) = \psi^{(p)}(S_0) \psi^{(q)}(T_0).\]
If $(p,g)P\bowtie G\cap (q,h)P\bowtie G=\emptyset$, then $pP\cap qP=\emptyset$ and thus $\psi^{(p)}(S_0) \psi^{(q)}(T_0)=0$ by the Nica-covariance of $\psi$. If $(p,g)P\bowtie G\cap (q,h)P\bowtie G=(r,k)P\bowtie G$, then $rP=pP\cap qP$ and thus 
\begin{align*}
    \Psi^{(p,g)}(S)\Psi^{(q,h)}(T) &= \psi^{(p)}(S_0) \psi^{(q)}(T_0) \\
    &= \psi^{(r)}(i_p^r(S_0) i_q^r(T_0)) \\
    &= \Psi^{(r,k)}(\iota_{r,k}(i_p^r(S_0) i_q^r(T_0))) \\
    &= \Psi^{(r,k)}(i_{(p,g)}^{(r,k)}(\iota_{p,g}(S_0)) i_{(q,h)}^{(r,k)}(\iota_{q,h}(T_0))) \\
    &= \Psi^{(r,k)}(i_{(p,g)}^{(r,k)}(S) i_{(q,h)}^{(r,k)}(T)). 
\end{align*}
Therefore, $\Psi$ is also Nica-covariant. The converse is clear since $\psi$ is the restriction of a Nica-covariant representation $\Psi$ on $X\cong \bigsqcup_{p\in P} Y_{(p,e)}$.
\end{proof}

\begin{definition} 
\label{D:NTXG}
Let $\cN\cT_X\bowtie G$ be the universal C*-algebra generated by the set of covariant representations $(\psi,U)$ of a \ZS system $(X,G,\beta)$ with $\psi$ Nica-covariant. 
\end{definition}

It follows from \cite[Proposition 6.8]{BLS2018b} and Example \ref{Eg:Fock} that $\cN\cT_X\bowtie G$ is always nontrivial. 

Completely similar to Corollaries \ref{C:HaoNgT} and \ref{C:HaoNgC}, we have the following Nica-Toeplitz type Hao-Ng isomorphism theorem. 

\begin{corollary} 
\label{C:HaoNgNT}
$\cN\cT_{X\bowtie G} = \cN\cT_X\bowtie G$.
\end{corollary}

%%%%%%%%%%%%%%%%%%%%%%%
%%%%%%%%%%%%%%%%%%%%%%%
\section{Examples} 

\label{S:EX}

In this section, we provide some examples of \ZS actions of groups $G$ on product systems $X=\bigsqcup_{p\in P} X_p$ and the associated C*-algebras.

\begin{eg}
Let $(X,G,\beta)$ be a \ZS system. 
\begin{itemize}
\item[(i)] If $P$ is trivial, then $\T_X \bowtie G\cong \O_X \bowtie G\cong\A\rtimes_\beta G$. 

\item[(ii)] If $G$ is trivial,  then $\T_X \bowtie G\cong\T_X$ and $\O_X \bowtie G\cong\O_X$. Furthermore, if $X$ is a compact aligned product system over a right LCM semigroup, then 
 $\N\T_X \bowtie G\cong\N\T_X$.
 \end{itemize}
\end{eg}

\begin{eg}
 Suppose that $(X,G,\beta)$ is a homogenous \ZS system.  If furthermore $g|_p=g$ for all $g\in G$ and $p\in P$, then 
 $X\bowtie G$ becomes the crossed product $X\rtimes G$, and  $\T_X\bowtie G$ (resp.~$\O_X\bowtie G$) is the crossed product $\T_X\rtimes G$ (resp.~$\O_X\rtimes G$). 
 So it follows from Corollary \ref{C:HaoNgT} (resp.~Corollary \ref{C:HaoNgC})  that we have the Hao-Ng isomorphisms: 
 $\T_X\rtimes G\cong T_{X\rtimes G}$ (resp. $\O_X\rtimes G\cong \O_{X\rtimes G}$).
 
Furthermore, if $P$ is right LCM and $X$ is compactly aligned,  then $\N\T_X\bowtie G$ is the crossed product $\N\T_X\rtimes G$, and so by
 the corresponding Hao-Ng isomorphism becomes  
 $
 \N\T_{X\rtimes G}\cong  \N\T_X\rtimes G
 $ 
 (cf. Corollary \ref{C:HaoNgNT}).
 \end{eg}

\begin{eg}
\label{Eg:gss}

Consider the trivial product system $X_P :=\bigsqcup_{p\in P} \bC v_p$ over $P$. 
%Recall that $(\lambda v_p)(\mu v_q)=\lambda\mu v_{pq}$ for all $\lambda, \mu\in \bC$ and $p,q\in P$. 
For $g\in G$, let 
\begin{align}
\label{E:tri}
\beta_g(\lambda v_p)=\lambda v_{g\cdot p} \qforal \lambda\in\bC\text{ and } p\in P.
\end{align}
It is easy to check that $\beta$ is a \ZS action of $G$ on $X_P$. Then we obtain the \ZS product system $X_P\bowtie G$ by Theorem \ref{T:ZSP}.

Suppose that $P$ is right LCM. By \cite[Theorem 4.3]{BRRW} %(resp. \cite[Theorem 5.2]{BRRW}),
one has 
\[
 \N\T_X \bowtie G\cong \ca(P\bowtie G). %\ (\text{resp. }\N\O_X \bowtie G\cong \Q(P\bowtie G)).
 \]
%%%%%%%%%
%%%%%%%%%
 Thus, in this case, by the Hao-Ng isomorphism theorems in Section \ref{S:main} we have 
 \begin{align*}
\ca(P)\bowtie G\cong \N\T_X \bowtie G\cong \N\T_{X \bowtie G}\cong \ca(P\bowtie G).
 \end{align*}
Let us remark that the above covers the examples studied in \cite{BRRW}. 
 \end{eg}

%%%%%%%%%%%%%%%about k-graph self-similar actions
\begin{eg}
\label{Eg:resss}
Let $G$ be a group and $\Lambda$ be a row-finite $k$-graph. Let $(G,\Lambda)$ be a self-similar $k$-graph (\cite{LY-IMRN, LY-JFA})
 such that for each $g\in G$ %, $g|_v=g$ ($v\in \Lambda^0$), and  % and $p\in \bN^k$
\begin{align}
\label{E:res}
d(\mu)=d(\nu)\implies g|_\mu=g|_\nu.
\end{align}
Then one can construct a \ZS product 
$\bN^k\bowtie G$ as follows. For $p\in \bN^k$, take $\mu \in \Lambda^p$ and define 
\[
g\cdot p := p,\ g|_p:=g|_\mu. 
\]

Let $X(\Lambda)=\bigsqcup_{p\in \bN^k} X(\Lambda^p)$ be the $C(\Lambda^0)$-product systems over $\bN^k$ associated to $\Lambda$ (see \cite{RS05} for all related details). Define 
\[
\beta_g: X(\Lambda^p) \to X(\Lambda^p), \quad \beta_g(\chi_\mu)=\chi_{g\cdot\mu}
\]
for all $g\in G$ and $\mu \in \Lambda^p$. 
Then $(X(\Lambda), G, \beta)$ is a \ZS system. Indeed, it suffices to check that it satisfies (A5) and (A6). For (A5), let $\mu\in \Lambda^p$ and $\nu\in\Lambda^q$. Then
\begin{align*}
\beta_g(\chi_\mu \chi_\nu)
&=
\delta_{s(\mu),r(\nu)}\, \chi_{g\cdot(\mu\nu)} 
=\delta_{s(\mu),r(\nu)}\chi_{g\cdot\mu (g|_\mu\cdot \nu)} 
=\delta_{s(g\cdot\mu),r(g_\mu\cdot \nu)}\, \chi_{g\cdot\mu} \chi_{(g|_\mu\cdot \nu)} \\
&=\beta_g(\chi_\mu)\beta_{g|_\mu}(\nu)
=\beta_g(\chi_\mu)\beta_{g|_p}(\nu).
\end{align*}
For (A6), let $\mu,\nu\in \Lambda^p$ and $v\in \Lambda^0$. We have 
\begin{align*}
\langle \chi_{g\cdot \mu}, \chi_{g\cdot \nu}\rangle(v) 
&= \sum_{v=r(\lambda)} \chi_{g\cdot \mu}(\lambda) \chi_{g\cdot \nu}(\lambda)
= |\{g\cdot \mu = g\cdot \nu, r(g\cdot \mu)=v\}|\\
%&(=\delta_{g\cdot \mu, g\cdot \nu}\delta_{v, r(g\cdot \mu})\\
&=|\{\mu = \nu, g\cdot r(\mu)= v\}|
  =|\{\mu = \nu, g|_\mu\cdot r(\mu)= v\}|\\
&=|\{\mu = \nu, r(\mu)=(g|_\mu)^{-1}\cdot v\}|
  =\langle \chi_{\mu}, \chi_{\nu}\rangle((g|_\mu)^{-1}\cdot v) \\
&=\beta_{g|_\mu}(\langle \chi_\mu, \chi_\nu\rangle)(v)
   =\beta_{g|_p}(\langle \chi_\mu, \chi_\nu\rangle(v).
\end{align*}

Therefore, by Theorem \ref{T:ZSP}, we obtain a \ZS product $X(\Lambda)\bowtie G$ over $\bN^k\bowtie G$. Also one can see that 
$\O_{X(\Lambda)\bowtie G}\cong \O_{X(\Lambda)}\bowtie G \cong \O_{G, \Lambda}$, 
where $\O_{G,\Lambda}$ is the self-similar $k$-graph C*-algebra associated with $(G,\Lambda)$ (\cite{LY-IMRN}). 
\end{eg}

\begin{rem}
(i) One can easily see that the condition \eqref{E:res} is equivalent to $g|_\mu=g|_\nu$ for all \textit{edges} $\mu,\nu\in\Lambda$.

(ii) At first sight, the condition \eqref{E:res} seems restrictive. But it is not hard to find self-similar $k$-graphs satisfying \eqref{E:res}. For example, let 
$g|_\mu=g|_\nu=:g$. Also, one can invoke \cite[Lemma 3.6]{LY-IMRN} to obtain more examples. 
\end{rem}

%Notice that Example \ref{Eg:1vertex} can not directly follow from Example \ref{Eg:resss} as the restriction \eqref{E:res} is not required in Example \ref{Eg:1vertex}. 

%%%%%%%%%%%revised 
We finish the paper by the following example, which does not really belong to \ZS actions considered in this paper. But it presents another natural way to construct a product system over $\bN^k$ which is isomorphic to the product system $X_{G,\Lambda}$ defined in \cite[Section~4]{LY-IMRN}. This provides one of our motivations of considering homogeneous actions in Section~\ref{S:homog}, and so we decide to include 
it here. 

\begin{eg}
Let $(G,\Lambda)$ be a self-similar $k$-graph. Let $X(\Lambda)=\bigsqcup_{p\in \bN^k} X(\Lambda^p)$ be the product systems over $\bN^k$ associated to $\Lambda$ as above. Define 
\[
\beta_g: X(\Lambda^p) \to X(\Lambda^p), \quad \beta_g(\chi_\mu)=\chi_{g\cdot\mu}. 
\]
Then one can easily check that
\begin{enumerate}
\item[(A1)$'$] for each $p\in P$, $\beta_g: X(\Lambda^p) \to X(\Lambda^p)$ is a $\mathbb{C}$-linear bijection;  
\item[(A2)$'$]  for any $g,h\in G$, $\beta_g\circ\beta_h=\beta_{gh}$;
\item[(A3)$'$]  the map $\beta_{e}$ is the identity map;
\item[(A4)$'$] the map $\beta_g$ is a $*$-automorphism on $X(\Lambda^0)=C(\Lambda^0)$; 
\item[(A5)$'$] for $p,q\in P$, $\mu\in \Lambda_p$ and $\nu\in \Lambda_q$, 
\[
\beta_g(\chi_\mu \chi_\nu)=\beta_g(\chi_\mu)\beta_{g|_\mu}(\chi_\nu);
\] 
%%%
\item[(A6)$'$] for $p\in P$ and $\mu,\nu\in \Lambda_p$, 
\[e
\langle\beta_g(\chi_\mu), \beta_g(\chi_\nu)\rangle=\beta_{g|_\mu}(\langle \chi_\mu, \chi_\nu\rangle).
\] 
\end{enumerate}

It follows from (A1)$'$-(A4)$'$ that $(C(\Lambda^0), G)$ is a C*-dynamical system. Let 
$\fA:=C(\Lambda^0)\rtimes_\beta G$. 
In what follows, we sketch the construction of an $\fA$-product system $\cE$ over $\bN^k$, which is isomorphic to the product system $X_{G,\Lambda}$ defined in \cite[Section~4]{LY-IMRN}. The details are similar to those in Section~\ref{S:homog} above and left to the interested reader. 
For $p\in \bN^k$, we construct a C*-correspondence $\E_p$ over $\fA$. 
For $x= \chi_v u_h$, $\xi=\chi_\mu u_g$ and $\eta=\chi_\nu u_h$ with $\xi_\mu, \eta_\nu\in X_p$,
define 
\begin{align*}
\xi x := \delta_{s(\mu), g\cdot v} \chi_\mu u_{gh},\
x\xi :=  \delta_{v, h\cdot r(\mu)} \chi_{h\cdot \mu} u_{h|_\mu g},\
\langle \xi, \eta\rangle 
:=u_{g^{-1}}\langle \chi_\mu,  \chi_\nu\rangle u_h.
\end{align*}
Let $\cE_p$ be the closure of the linear span of $\chi_\mu u_g$ with $\mu \in \Lambda^p$ and $g\in G$.
Set 
\[
\cE:=\bigsqcup_{p\in \bN^k} \cE_p.
\] 
For $\xi=\sum \chi_\mu u_g\in \cE_p$ and $\eta=\sum \chi_\nu u_h\in \cE_q$ define
\begin{align*}
\big(\sum \chi_\mu u_g\big)\big(\sum \chi_\nu u_h\big):=\sum \chi_\mu \chi_{g\cdot \nu} u_{g|_\nu} u_h=\sum \delta_{s(\mu), r(g\cdot \nu)}\chi_{\mu g\cdot \nu} u_{g|_\nu h}.
\end{align*}
Then $\cE$ is an $\fA$-product system over $\bN^k$.
One can check that $\cE$ is isomorphic to the product system $X_{G,\Lambda}$ defined in \cite[Section~4]{LY-IMRN} via the map
\[
%Z\to X_{G,\Lambda}, \ 
\chi_\mu u_g\in \cE_p\mapsto \chi_\mu j(g)\in X_{G,\Lambda,p}.
\]
\end{eg}

%\bibliographystyle{abbrv}
%\bibliography{semigroup}

\end{document}